\documentclass[12pt]{amsart}

\usepackage{hyperref}
\usepackage{amssymb}
\usepackage{amsmath}
\usepackage{amsthm}
\usepackage{enumerate}
\usepackage{graphicx}
\usepackage{mathrsfs}
\usepackage{color}
\numberwithin{equation}{section}

\usepackage{setspace}

\theoremstyle{plain}
\newtheorem{thm}{Theorem}[section]
\newtheorem{lem}[thm]{Lemma}
\newtheorem{prop}[thm]{Proposition}
\newtheorem{cor}[thm]{Corollary}

\theoremstyle{definition}
\newtheorem{de}{Definition}[section]
\newtheorem{ex}{Example}[section]

\theoremstyle{remark}
\newtheorem{re}{Remark}[section]

\begin{document}
	
\title[Lower bounds for uncentered maximal functions on metric measure space]
	{Lower bounds for uncentered maximal functions on metric measure space}
	
\author{Wu-yi Pan and Xin-han Dong$^{\mathbf{*}}$}

\address{Key Laboratory of Computing and Stochastic Mathematics (Ministry of Education),
		School of Mathematics and Statistics, Hunan Normal University, Changsha, Hunan 410081, P.
		R. China}
	
\email{pwyyyds@163.com}
\email{xhdonghnsd@163.com}

\date{\today}
	

\thanks{The research is supported in part by the NNSF of China (Nos. 11831007, 12071125)}	
	
\thanks{$^{\mathbf{*}}$Corresponding author}
	
	
\begin{abstract}
%Let $(X,d,\mu)$  be a metric measure space satisfying the Besicovitch covering property. We show that the uncentered Hardy-Littlewood maximal operators associated with $\mu$ have the uniform lower  $L^p$-bounds (independent of $\mu$) that are strictly greater than $1$, if $\mu$ satisfies a natural continuity assumption and $\mu(X)=\infty$. In addition, the same being false for measures without continuity conditions, we can construct the counterexamples where the lower bounds can be close to $1$ arbitrarily.
We show that the uncentered Hardy-Littlewood maximal operators associated with the Radon measure  $\mu$ on $\mathbb{R}^d$ have the uniform lower  $L^p$-bounds (independent of $\mu$) that are strictly greater than $1$, if  $\mu$ satisfies a mild continuity assumption and $\mu(\mathbb{R}^d)=\infty$. We actually do that in the more general context of metric measure space  $(X,d,\mu)$ satisfying the Besicovitch covering property. 
In addition, we also illustrate that the continuity condition can not be ignored by constructing counterexamples.

%In addition, we can construct counterexamples so that the continuity condition can not be ignored.

%AaThe conclusion fails in general for removing the continuity condition, and indeed 
\end{abstract}
	
\maketitle
\section{ Introduction } 

Let $Mf$ be the uncentered Hardy-Littlewood maximal function of $f$, where $f$ is a $p$-th power Lebesgue integrable function for $p>1$. Maximal function plays at least two roles. Firstly it provides an important example of a sub-linear operator used in real analysis and harmonic analysis. Secondly, $Mf$ is apparently comparable to the original function $f$ in the $L^p$ sense. Not only that, by Riesz's sunrise lemma, Lerner \cite{Le10} proved for the real line that \begin{equation*}
	\Vert Mf\Vert_{L^p(\mathbb{R})} \geq \left(\frac{p}{p-1}\right)^{\frac{1}{p}}\Vert f\Vert_{L^p(\mathbb{R})}.
\end{equation*}
And Ivanisvili et al. found the lower bound of the high dimensional Euclidean space in \cite{IJN17}. 

One may suspect whether in general metric measure space $(X,d,\mu)$, for given $1<p<\infty$, there exists a constant $\varepsilon_{p,d,\mu}>0$ such that 
\begin{equation*}
	\Vert M_\mu f\Vert_{L^p(\mu)} \geq (1+\varepsilon_{p,d,\mu})\Vert f\Vert_{L^p(\mu)}, \quad\text{for all $f\in L^p(\mu)$},
\end{equation*}
where $M_\mu$ (cf. \eqref{de:max} below for the definition) is the uncentered Hardy-Littlewood maximal operator associated with $\mu$.
Unfortunately, such a measure must not be finite. In fact, $1_{X}$ is a fixed point of $M_\mu$ in $L^p(\mu)$ if $\mu(X)< \infty$. This suggests that we only need to consider infinite measures.

One of our main results is the following.

\begin{thm}\label{thm:1.1}
	Let $(X,d,\mu)$  be a metric measure space satisfying the Besicovitch covering property with constant $L$. Suppose that \begin{equation}\label{con:1.1}
		\mu(\{x\in \operatorname {supp} (\mu): r\mapsto \mu(B(x,r))\; \text{is discontinuous}\})=0
	\end{equation} and $\mu(X)=\infty$. Then 
	\begin{equation}\label{ineq:main}
		\Vert M_\mu f \Vert_p \geq \left(1+\frac{1}{(p-1)L}\right)^{\frac{1}{p}}\Vert f\Vert_p, \quad \text{for all}\;\; f\in L^p(\mu).
	\end{equation}
\end{thm}

Since any normed space  $(\mathbb{R}^d,\Vert\cdot\Vert)$ has the Besicovitch covering property with a constant which equals its strict Hadwiger number $H^*(\mathbb{R}^d,\Vert\cdot\Vert)$, we derive

\begin{cor}\label{cor:1.2}
		Let $\Vert\cdot\Vert$ be any norm on $\mathbb{R}^d$. Let $\mu$ be a Radon measure on $(\mathbb{R}^d,\Vert\cdot\Vert)$ such that
	$\mu(\mathbb{R}^d)=\infty.$ Suppose that \begin{equation*}
		\mu(\{x\in \operatorname {supp} (\mu): r\mapsto \mu(B(x,r))\; \text{is discontinuous}\})=0.
	\end{equation*} Then 
	\begin{equation*}
		\Vert M_\mu f \Vert_p \geq \left(1+\frac{1}{(p-1)H^*(\mathbb{R}^d,\Vert\cdot\Vert)}\right)^{\frac{1}{p}}\Vert f\Vert_p,\quad \text{for all $f\in L^p(\mu)$}.
	\end{equation*}
\end{cor}

Now let us make a few remarks about the condition of \eqref{con:1.1}. By Lebesgue decomposition theorem, any Radon measure can be divided into three main parts: the absolutely continuous part; the singular continuous part; the discrete measure part. If its discrete measure part exists (that is, measure is not continuous or has atoms), then it must not satisfy \eqref{con:1.1}. But if it contains only the absolutely continuous part, the condition is true. 

Indeed, Theorem \ref{thm:1.1} does not hold in general for infinite Radon measures. We can construct a measure containing atoms to show that \eqref{ineq:main} is invalid.

On the other hand, if $d=1,2$, then \eqref{con:1.1} is satisfied for every Radon measure on the Euclidean space $(\mathbb{R}^d,\Vert\cdot\Vert_2)$ having no atoms. Hence for this metric measure space, we have $\Vert M_\mu f\Vert_p >\Vert f\Vert_p$ for all $f\in L^p$. More precisely, we prove

\begin{thm}\label{thm:1.3}
	If $\mu$ is a Radon continuous measure  on $(\mathbb{R}^2,\Vert\cdot\Vert_2)$ such that $\mu(\mathbb{R}^2)=\infty$, then 	
	\begin{equation*}
		\Vert M_\mu f \Vert_p \geq \left(1+\frac{1}{5(p-1)}\right)^{\frac{1}{p}}\Vert f\Vert_p,\quad\text{for all $f\in L^p(\mu)$}.
	\end{equation*}
\end{thm}

%Nevertheless, it is false in three dimension or above. We can even find a singular continuous measure that the lower operator norm of $M_\mu f$ arbitrarily approaches to $1$.

It is worth noting that we formulate and prove Lerner's theorem in a more general setting which remains true. 

\begin{thm}\label{thm:1.4}
	If $\mu$ is a Radon continuous measure  on $(\mathbb{R},\Vert\cdot\Vert_2)$ such that $\mu(\mathbb{R})=\infty$, then 	
	\begin{equation*}
		\Vert M_\mu f \Vert_p \geq \left(\frac{p}{p-1}\right)^{\frac{1}{p}}\Vert f\Vert_p,\quad\text{for all $f\in L^p(\mu)$}.
	\end{equation*}
\end{thm}
Besides, the lower $L^p$-bounds of the maximal operator with respect to a more general measure were discussed and we extended the results of Theorem \ref{thm:1.4} to a class of measures containing only one or two  atoms. The proof depends on our adopted approach from a variant of Theorem \ref{thm:1.1},  which differs from the traditional one of Theorem \ref{thm:1.4}. It also means that the condition of \eqref{con:1.1} is not necessary.

The paper is organized as follows.
In Section \ref{S2}, we give some definitions and basic properties of metric measure spaces $(X,d,\mu)$ and the corresponding maximal operators. In Section \ref{S3}, we  prove Theorem \ref{thm:1.1} and Corollary \ref{cor:1.2}. In Section \ref{S4}, we mainly restrict our attention to the case of Radon measures in Euclidean space; the counterexample and other theorems mentioned above are presented.
$\\ \hspace*{\fill} \\$

\section{Preliminaries}\label{S2}

Throughout the  whole paper, we use the same notation as \cite{Al21}. If not declared specifically, the symbol $B^{cl}(x,r):=\{y\in X: d(x,y)\leq r\}$ denotes closed balls, and $B^o(x,r):=\{y\in X: d(x,y)< r\}$ to refer to open balls.   Because most of the definitions in this paper do not depend on the selection of open and closed balls, we generally use $B(x,r)$ to denote it, but it should be noted that all balls are taken to be of the same kind when we utilize $B(x,r)$ once. When referring to $B$, suppose that its radius and center have been determined.

\begin{de}
	A measure on a topological space is Borel if it is defined on the $\sigma$-algebra generated by all open sets. 
\end{de}
Sometimes we need to study problems on a $\sigma$-algebra larger than Borel $\sigma$-algebra, so we definite Borel semi-regular measure.
\begin{de}
	A measure space $(X,\mathcal{A},\mu)$ on a topological space is Borel semi-regular measure if for every $A\in
	 \mathcal{A}$, there exists a Borel set $B$ such that $\mu((A\backslash B)\bigcup (B\backslash A))=0$.
\end{de}
In general, the maximal operator problem studied in this paper has no essential difference between these two kinds of measures. Therefore, although our conclusion is valid for general Borel semi-regular measure, the proofs are always carried out under the assumption of Borel measure. We will explain this reason in the proof of Lemma \ref{lem:3.4}.

\begin{de}\cite[Definition 7.2.1]{Bo07}
	A Borel measure is $\tau$-additive, if for every collection $\mathcal{B}=\{U_{\lambda}:\lambda\in \Lambda\}$ of open sets, 
	\begin{equation*}
		\mu\left(\bigcup_{\lambda}U_\lambda\right)=\sup_{\mathcal{F}}\mu\left(\bigcup_{i=1}^nU_{\lambda_i}\right)
	\end{equation*} 
where the supremum is taken over all finite subcollections $\mathcal{F}\subset \mathcal{B}$. A Borel measure on a metric space is locally finite if there exists an $r > 0$ such that $\mu(B(x,r))<\infty$ for every $x\in X$.
\end{de}
Recall that the complement of the support of the Borel measures $(\operatorname {supp} (\mu))^c:=\bigcup\{B(x,r):x\in X,\mu(B(x,r))=0\}$ is open, and hence measurable. If $\mu$ is $\tau$-additive, we immediately know from its definition that $\mu$ is full support, i.e., $\mu(X\backslash\operatorname {supp} (\mu))=0$. 

We also note that if a metric space is separable, then it is second countable, and hence $\tau$-additivity holds for all Borel measures. 

In particular, locally finite  Borel regular measures
in complete separable metric spaces are equivalent to Radon measures, see e.g. Schwartz \cite[Part I, §11.3]{Sc73}. Note that our definition of locally finite is different from that of Aldaz \cite[Definition  2.1]{Al21}, and we need to identify carefully. Locally finiteness in the sense of Aldaz refers to that if each bounded set has finite measure.

\begin{de}
	We say that $(X,d,\mu)$ is a metric measure space if $\mu$ is a Borel semi-regular measure and its restriction on Borel set is $\tau$-additive and locally finite. 
\end{de}

Such a measure is undoubtedly $\sigma$-finite. A function $f\in L^{1}_{\operatorname {loc}}(\mu)$ is defined to be local integral if $\int_B|f|d\mu<\infty$ for each ball $B$. 
For $1\leq p<\infty$, we define 
$
	L^{p}_{\operatorname {loc}}(\mu):=\{f:|f|^p\in L^{1}_{\operatorname {loc}}(\mu)\}.
$  

Recall the  \emph{uncentered Hardy-Littlewood maximal operator} acting on a locally integrable function $f$ by
\begin{equation}\label{de:max}
	M_{\mu}f(x):=\sup_{x\in  B;\mu(B)>0}\frac{1}{\mu(B)}\int_{B}|f|d\mu, \quad x\in X.
\end{equation} We  will denote $M_\mu f$ briefly by  $Mf$ when no confusion can arise. It may be checked that for any $f\in L^{1}_{\operatorname {loc}}(\mu)$, $Mf$ is lower semicontinuous, hence it is Borel measurable. By approximation, it is insignificant in the definition whether one takes the balls $B(x,r)$ to be open or closed. These were shown by Stempak and Tao (see \cite[Lemma  3]{ST14}).

Now we are in the position of the definition of Besicovitch covering property. 

\begin{de}
	 We say that $(X,d)$ satisfies the Besicovitch covering property (BCP) if there exists a constant
	$L\in \mathbb{N}^+$ such that for every $R>0$, every set $A\subset X$, and every cover $\mathcal{B}$ of $A$ given by
	\begin{equation*}
		\mathcal{B}=\{B(x,r):x\in A, 0<r<R\},
	\end{equation*} then there is a countable subfamily $\mathcal{F}\subset \mathcal{B}$ such that the balls in $\mathcal{F}$ cover $A$, and every point in $X$ belongs to at most $L$
	balls in $\mathcal{F}$, that is,
\begin{equation*}
		\mathbf{1}_{A}\leq \sum_{B(x,r)\in \mathcal{F}}\mathbf{1}_{B(x,r)}\leq L.
\end{equation*}
\end{de}
Note that unlike \cite{Al19}, we require subfamily  to be countable. Recall that the validity of BCP is sufficient to imply the validity of the differentiation theorem for every locally finite Borel semi-regular measure. See for instance \cite{Ma95,LR17,LR19}.
$\\ \hspace*{\fill} \\$

\section{The proof of the main result}\label{S3}
We shall always work on a metric measure space $(X,d,\mu)$. We denote by $\langle f\rangle_A$ the integral average of $f$ over a measurable set $A$, namely $\langle f\rangle_A =\frac{1}{\mu(A)}\int_A|f|d\mu$. If $\mu(A)=0$ then we set $\langle f\rangle_A =0$.

To prove the main result, we need to establish the following lemmas.

\begin{lem}\label{Lem:3.1}
	Let $f,f_n\in L^{1}_{\operatorname {loc}}(\mu)$ be non-negative. If $\liminf\limits_{n\to \infty}\int_{B}f_nd\mu\geq\int_{B}fd\mu$ for all balls $B$ with $\mu(B)<\infty $, then $\liminf\limits_{n\to \infty}Mf_n(x)\geq Mf(x)$.
\end{lem}

\begin{proof}
	Fix a point $x\in X$. If $Mf(x)< \infty$, so for every real number $\varepsilon>0$, there exists a ball such that $\mu(B_\varepsilon)>0$ and $\langle  f\rangle_{B_{\varepsilon}}>Mf(x)-\frac{\varepsilon}{2}$. By the assumption that $\liminf\limits_{n\to \infty}\int_{B}f_nd\mu\geq\int_{B}fd\mu$ for all balls $B$, so we have
	$\liminf\limits_{n\to \infty}\langle f_n\rangle_{B_{\varepsilon}}\geq\langle f\rangle_{B_{\varepsilon}}$, then for $\varepsilon>0$,  there exists a natural number $N_{\varepsilon}>0$ such that $\langle f_n\rangle_{B_{\varepsilon}}\geq\langle f\rangle_{B_{\varepsilon}}-\frac{\varepsilon}{2}$ for ${\displaystyle n\geq N_{\varepsilon},}$ hence $\langle f_n\rangle_{B_{\varepsilon}}>Mf(x)-\varepsilon$. Applying the definition of $Mf_n(x)$, for $n> N_{\varepsilon}, $ we get $Mf_n(x)>Mf(x)-\varepsilon$, which prove the lemma in the case of  $Mf(x)< \infty$. 
	
	Now suppose $Mf(x)=\infty$. Thus for every $M>0$, we  can also find a ball $B_M$ such that $\mu(B_M)>0$ and $\langle  f\rangle_{B_M}>M$. The same way shows that there exists a $N_M>0$ such that $\langle f_n\rangle_{B_M}>\frac{M}{2}$ for all $n\geq N_M$. Hence  $\liminf\limits_{n\to \infty}Mf_n(x)=\infty$. This completes the proof.
\end{proof}

\begin{cor}\label{cor:3.2}
	Let $f,f_n\in L^{1}_{\operatorname {loc}}(\mu)$ be non-negative. If $f_n$ is monotonically increasing and  converges to $f$ a.e., then $\lim\limits_{n\to \infty}\Vert f_n\Vert_p=\Vert f\Vert_p$ and $\lim\limits_{n\to \infty}\Vert Mf_n\Vert_p=\Vert Mf\Vert_p$ 
\end{cor}
\begin{proof}
	Since strong convergence implies weak convergence, $\lim\limits_{n\to\infty}\int_{B}f_nd\mu=\int_{B}fd\mu$ for all balls $B$ with $\mu(B)<\infty$. By Lemma \ref{Lem:3.1},  we have $ Mf \leq \liminf\limits_{n\to \infty}Mf_n$. Further, operator $M$  is order preserving, then $Mf_n$ is  monotonically increasing  and converges to $Mf$. By monotone convergence theorem, the result follows.
\end{proof}
The following approximation theorem is well known (see
e.g. \cite[Theorem 1.1.4]{EG92},  \cite[Theorem 2.2.2]{Fe69} or \cite[Theorem 1.3]{Si83}).

\begin{lem}\label{lem:3.3}
	Let $\mu$ be a Borel semi-regular measure on $(X,d)$, $E$ a $\mu$-measurable set, and $\varepsilon>0$. If $\mu(E)<\infty$, then there is a bounded closed set $C\subset E$ such that $\mu(E\backslash C)\leq \varepsilon$.
\end{lem}

Recall that a  finitely simple Borel function  has the form
$\sum\limits_{i=1}^Nc_i\mathbf{1}_{E_i}$,
where $N<\infty, c_i\in \mathbb{R}$, and $E_i$ are pairwise disjoint Borel sets with $\mu(E_i)<\infty$. The support of a measurable function $g$, denoted by $\operatorname {supp} (g)$, is the closure of the set $\{x\in X:g(x)\neq 0\}$.
\begin{lem}\label{lem:3.4}
	Let $C$ be a constant.  The following are equivalent: \begin{enumerate}[\rm(i)]
		\item For all $f\in L^p(\mu)$, $\Vert Mf\Vert_p\geq C \Vert f\Vert_p$.
		\item For all non-negative finitely simple Borel function $g$, we have $\Vert Mg\Vert_p\geq C\Vert g\Vert_p$.
		\item For all non-negative bounded upper semi-continuous functions $g$ whose $\operatorname {supp} (g)$ is bounded  and $ \mu(\operatorname {supp} (g))< \infty$,  we have $\Vert Mg\Vert_p\geq C \Vert g\Vert_p$.
\end{enumerate}
\end{lem}
	\begin{proof}
		Without loss of generality, the functions that appear in this proof are all non-negative.
		By restricting the measure on its Borel $\sigma$-algebra to get a new $\nu$, for $f\in L^p(\mu)$, there is a Borel function  $g$ such that $\mu$-a.e. $g=f$. Thus \begin{equation*}
		M_\mu f =\sup_{x\in  B;\mu(B)>0}\frac{\int_{B}fd\mu}{\mu(B)}=\sup_{x\in  B;\mu(B)>0}\frac{\int_{B}gd\mu}{\nu(B)}=\sup_{x\in  B;\mu(B)>0}\frac{\int_{B}gd\nu}{\nu(B)}=M_\nu g
	\end{equation*} and $\Vert g\Vert_{p,\nu}= \Vert f\Vert_{p,\mu}$. This means that we only need to focus on  Borel functions.
	
	Now we are in the position that $\textrm{(ii)}$ implies $\textrm{(i)}$. Applying Corollary \ref{cor:3.2},  it suffices to prove that for every Borel function $f\in L^p(\mu)$,  there exists a finitely simple Borel function sequence $\{f_{n}\}_{n=1}$ which is monotonically increasing and  converges to $f$. In fact, since $f\in L^p(\mu)$, $f$ is a.e. finite. Recall that $\mu$ is $\sigma$-finite and let the pairwise disjoint subsets $A_i$ such that $\bigcup_{i=1}^\infty A_i=X$ and $\mu(A_i)<\infty$. For $k\in \mathbb{N}$, set 
	\begin{equation*}
		E_{k,j}=\{x\in \bigcup_{i=0}^kA_i:\frac{j-1}{2^k}\leq f(x)<\frac{j}{2^k}\}\; (j=1,2,...,k\cdot 2^k)
	\end{equation*}
	and 	\begin{equation*}
		f_k(x)=\sum_{j=1}^{k\cdot 2^k}\frac{j-1}{2^k}\mathbf{1}_{E_{k,j}}(x).
	\end{equation*}
	Clearly, $f_k$ is what we need by simple function approximation theorem.
	
	As we have shown above, it suffices to consider the  simple Borel function $f=\sum_{n=1}^Nc_n\mathbf{1}_{B_n}$, where $B_n$ is a  Borel set with $\mu(B_n)<\infty$.  By Lemma \ref{lem:3.3}, for each $\varepsilon>0$ and  $1\leq n\leq N$ there exists a bounded closed set $F_{n,\varepsilon}$ with $F_{n,\varepsilon}\subset B_n$ and $c_n\mu(B_n\backslash F_{n,\varepsilon})<\frac{\varepsilon}{2^{N}}$. Set $\psi_k =\sum_{n=1}^Nc_n\mathbf{1}_{F_{n,\frac{1}{k}}}.$ Then $\psi_k$ is upper continuous and supports on a bounded closed set with $ \mu(\operatorname {supp} (\psi_k))< \infty$, and $\psi_k \uparrow f$ as desired.
\end{proof}

\begin{lem}\label{lem:3.5}
	Let $\mu$ be an infinite Borel semi-regular measure on $(X,d)$. If the sequence $x_{n}\in X$ is bounded and $\lim\limits_{n\to \infty}r_n =\infty$, then $\lim\limits_{n\to \infty}\mu(B(x_n,r_n))= \infty$.
\end{lem}
\begin{proof}Suppose that $x_{n}\in B(x,R)$ for some $x\in X$. Applying the continuity from above of measure, we get $\lim\limits_{n\to \infty}\mu(B(x,\frac{r_n}{2})) =\mu(X)=\infty.$ Since $\lim\limits_{n\to \infty}r_n =\infty$, there is a $N$ such that $B(x,\frac{r_n}{2})\subset B(x_n, r_n)$ for every $n\geq N$. Hence $\lim\limits_{n\to \infty}\mu(B(x_n,r_n))=\infty$ which proves lemma.
 \end{proof}

Now we follow methods from \cite[Theorem 1.1]{IJN17} to prove Theorem \ref{thm:1.1}.

\begin{proof}[Proof of Theorem \ref{thm:1.1}.]
		 By Lemma \ref{lem:3.4}, it remains to prove that \eqref{ineq:main} holds for all non-negative bounded  functions $f$ whose $\operatorname {supp} (f)$ is bounded and $ \mu(\operatorname {supp} (f))< \infty$. Suppose first that $f$ is bounded by $C$. Then 
		 $\int_{X}fd\mu \leq C\mu(\operatorname {supp} (f))< \infty$, and hence $f\in L^1(\mu)$. 
		 
		 As $(X,d)$ satisfies the Besicovitch covering property, the differentiation theorem holds, hence $\lim\limits_{r\to 0}\langle f\rangle_{B(x,r)} = f(x)$ almost everywhere $x\in X$ for every $f\in  L^{1}_{\operatorname {loc}}(\mu)$. If we set $F=\{x\in X: \lim\limits_{r\to 0}\langle f\rangle_{B(x,r)} = f(x)\}$, then $\mu(X\backslash F)=0$. Let \begin{equation*}
		 	E=\{x\in \operatorname {supp} (\mu): r\mapsto \mu(B(x,r))\; \text{is continuous}\}.
		 \end{equation*} Clearly, $\mu(X\backslash E)=0$ by assumption. Fix $t>0$ and consider $K_{t}=\{x\in  E\bigcap F: f(x)> t\}$. For fixed $x\in K_t$, applying the definition of $ \operatorname {supp} (\mu)$ and $F$, now choose a ball $B$ centered at $x$ such that  $\langle f\rangle_{B}>t$ and $\mu(B)>0$.	From the assumption that $\mu(X)=\infty$, we obtain $\lim\limits_{n\to \infty}\langle f\rangle_{B(x,r_{n})}=0$ as $\lim\limits_{n\to \infty}r_n =\infty$. Fix $x\in K_t$ and set $p(s)=\langle f\rangle_{B(x,s)}$. Then $p(s)$ is a continuous function on $(0,\infty]$ that has the intermediate value property and so $\{r: \langle f\rangle_{B(x,r)}=t\}$ is non-empty for every $x\in K_t$. Let 
	 	\begin{equation*}
		 R_t=\sup\limits_{x,r}\{r:x\in K_t, \langle f\rangle_{B(x,r)}=t\}.
	 	\end{equation*}
  		We show that $R_t<\infty$. To do it, we argue by contradiction. Suppose, if possible, that $R_t = \infty$. Then there exists a sequence $(x_n, r_n)\in K_t \times (0,\infty]$ such that  $\lim\limits_{n\to \infty} r_n = \infty$ and $\langle f\rangle_{B(x_n,r_n)}=t$. As $\operatorname {supp} (f)$ is bounded, Lemma \ref{lem:3.5} gives $\limsup\limits_{n\to \infty}\langle f\rangle_{B(x_{n},r_{n})}\leq \limsup\limits_{n\to \infty}\frac{\Vert f\Vert_1}{\mu(B(x_{n},r_{n}))}$ = 0. Thus $\lim\limits_{n\to \infty}\langle f\rangle_{B(x_{n},r_{n})}=0$ which obviously contradicts $t>0$. Note that this number depends only on $t$ if $f$ is given. Thus for $x\in K_t$, there exists an $r_x\leq R_t$ such that $\langle  f\rangle_{B(x,r_x)}=t$. Now for cover $\mathcal{C}$ of $K_t$ given by 
		\begin{equation*}
		\mathcal{C}=\{B(x, r) :x\in K_t,\;\langle f\rangle_{B(x,r)}=t\},
	\end{equation*}
applying Besicovitch covering property, we extract a countable subfamily $B(x_{t,j}, r_{t,j})\in \mathcal{C}$ so that 
	\begin{equation*}
		\mathbf{1}_{K_t} \leq \psi(x,t):=\sum_{j}\mathbf{1}_{B(x_{t,j}, r_{t,j})}\leq L.
	\end{equation*}

We show that $\psi(x,t)$ satisfies the following properties also: \begin{enumerate}[\quad \quad \quad \rm(1)]
	\item if $t>Mf(x)$ then $\psi(x,t)=0$;
	\item if $f(x)>t$, then $\psi(x,t)\geq 1$ almost everywhere;
	\item for every $t>0$, we have $\int_{X}t\psi(x,t)d\mu(x)=\int_{X}\psi(x,t)f(x)d\mu(x)$.
\end{enumerate}

 For the first property, we prove it by contradiction. Suppose, if possible, that $\psi(x,t)> 0$, then there exists a $B(x_{t,\ell},r_{t,\ell})$ containing $x$. Thus, $Mf\geq\langle f\rangle_{B(x_{t,\ell},r_{t,\ell})}=t$, contradicting the assumption that $t> Mf(x)$.
 
 To obtain the second property, let $x\in K_t$. The selection of $\psi$ gives $\psi(x,t)\geq \mathbf{1}_{K_t}(x)\geq 1$, and the property follows by $\mu(X\backslash(E\bigcap F))=0$.
 
 The third property follows immediately.
 
 We now seek to  prove the desired inequality. Since $\mu$ is $\sigma$-finite, applying Fubini-Tonelli's theorem, (3) implies the following equality: 
 \begin{equation*}
 	\int_{X}\int^{\infty}_0t^{p-1}\psi(x,t)dtd\mu(x)= 	\int_{X}\int^{\infty}_0t^{p-2}\psi(x,t)f(x)dtd\mu(x).
 \end{equation*}
By property (1), we can restrict the integration to $[0,Mf(x)]$, that is
 \begin{equation}\label{3.1}
	\int_{X}\int^{Mf(x)}_0t^{p-1}\psi(x,t)dtd\mu(x)= 	\int_{X}\int^{Mf(x)}_0t^{p-2}\psi(x,t)f(x)dtd\mu(x).
\end{equation}
Now, since $\psi\leq L$, and  $Mf\geq f$ a.e., it follows from the above equality \eqref{3.1} that
\begin{align*}
	\frac{L}{p}\left(\Vert Mf\Vert_p^p-\Vert f\Vert_p^p\right)&\geq \int_{X}\int^{Mf(x)}_0t^{p-1}\psi(x,t)dtd\mu(x)-\int_{X}\int^{f(x)}_0t^{p-1}\psi(x,t)dtd\mu(x)\\
	&\geq \int_{X}\int^{f(x)}_0t^{p-1}\psi(x,t)\left(\frac{f(x)}{t}-1\right)dtd\mu(x).
\end{align*}
Note that property (2) yields 
\begin{align*}
	\frac{L}{p}\left(\Vert Mf\Vert_p^p-\Vert f\Vert_p^p\right)&\geq\int_{X}\int^{f(x)}_0t^{p-1}\left(\frac{f(x)}{t}-1\right)dtd\mu(x)= \frac{\Vert f\Vert_p^p}{p(p-1)}.
\end{align*}
This finishes the proof.
\end{proof}
\begin{re} 
	More generally, the same is true for quasi-metric in place of metric if we assume that all balls are Borel semi-regular measurable. It may happen that a ball in quasi-metric space is not a Borel set. To avoid such pathological cases, the assumption  must be made to ensure the definition of $M_\mu$ is reasonable. In fact, Mac\'{ı}as and Segovia showed in  \cite[Theorem 2, p. 259]{MS79} that  there exists an $\alpha\in (0,1)$ and a quasi-metric $d_*$ which is equivalent to original quasi-metric and original topology is metrizable by $(d_*)^\alpha$. Thus we can generalize  Lemma \ref{lem:3.3} to quasi-metric space since the bounded closed sets of both are the same. Lemma \ref{lem:3.5} can also be established by quasi-triangle inequality. From the foregoing discussion, the same proof carries over into quasi-metric space.
\end{re}

%\begin{re}
%	Set $c_p=\left( 1+\frac{1}{(p-1)L}\right)^{1\slash p}$ and $\phi (t) =\frac{t^p}{(1+t)^p}$ for $t\in [0,\infty)$. Clearly, $\phi$ is  differentiable, increasing and $\phi(0)=0$. After multiplying  $\int_{X}t\psi(x,t)d\mu(x)=\int_X\psi(x,t)f(x)d\mu(x)$ both sides by $\phi^{'}(t)$,
%  a similar argument shows that 
%	\begin{equation*}
%		\int_{X}\frac{M^pf(x)}{(Mf(x)+t)^p}d\mu \geq c^p_p\int_{X}\frac{|f(x)|^p}{(|f(x)|+t)^p}d\mu +\left(\frac{c^p_p-1}{t}\right)\int_{X}\frac{|f(x)|^{p+1}}{(|f(x)|+t)^p}d\mu\;\;\text{}\;\; 
%	\end{equation*} for all $f\in L^p\cap L^{\infty}(\mu)$.
%\end{re}
%The important point in the preceding argument is that 

A set of the form $L^+_t(f):=\{x\in \operatorname {supp} (\mu): f(x)>t\}$ is called a \emph{strict superlevel sets} of $f$. It is rather straightforward to see that the following theorem remains valid by modifying the above proof. 

\begin{thm}\label{thm:3.6}
		Let $f\in L^{p}(\mu)$ be non-negative.  If $\mu(X)=\infty$, and for any $t>0$, there exists a finite or countable ball-coverings $\mathcal{F}_t$ of $L^+_t(f)$ such that the average value of the function $f$ on each ball $B\in \mathcal{F}_t$ is equal to $t$, and almost everywhere every point in $\operatorname {supp} (\mu)$ belongs to at most $L$
	balls in $\mathcal{F}_t$, then 
	\begin{equation*}
		\Vert M 
		f \Vert_p \geq \left(1+\frac{1}{(p-1)L}\right)^{\frac{1}{p}}\Vert f\Vert_p.
	\end{equation*}
	Especially, if the balls are  pairwise disjoint, then 	$
	\Vert Mf \Vert_p \geq \left(\frac{p}{p-1}\right)^{\frac{1}{p}}\Vert f\Vert_p.
	$
\end{thm}
Whether the theorem is true without Besicovitch covering property seems worthwhile to pursue. However, due to the bottleneck in the extraction of the ball subfamily, the approach of Theorem \ref{thm:1.1} can not help. To the surprise, if we preserve the continuity assumption of $\mu(B(x_0,r))$ and impose on $f$ the additional conditions of being  radially decreasing symmetric $($with the point $x_0$$)$, we can avoid such difficulties.

\begin{thm}
	Let $f\in L^{p}(\mu)$ be a radial decreasing function with the point $x_0\in X$.  If $\mu(X)=\infty$ and the function $r\mapsto \mu(B(x_0,r))$ is continuous on the interval $(0,+\infty)$, then 
	\begin{equation*}
		\Vert Mf \Vert_p \geq \left(\frac{p}{p-1}\right)^{\frac{1}{p}}\Vert f\Vert_p.
	\end{equation*}
\end{thm}
\begin{proof}
	Note first that the strict superlevel sets of a radial decreasing function are always balls centered on the point $x_0$. Then we can choose a larger ball $B$ centered at $x_0$ to contain $L^+_t(f)$ while making $\langle f\rangle_{B}=t$ possible for any $t$ by the assumption. Thus the theorem follows immediately applying Theorem \ref{thm:3.6}.
\end{proof}
Given a norm $\Vert\cdot\Vert$ on $\mathbb{R}^d$,  \emph{the strict Hadwiger number} $H^*(\mathbb{R}^d,\Vert\cdot\Vert)$ is the maximum number of translates of the closed unit
ball $B^{cl}(0,1)$ that can touch $B^{cl}(0,1)$ and such that any two  translates
are disjoint.  See \cite[p. 10]{Al21} in more detail.
\begin{proof}[Proof of Corollary \ref{cor:1.2}.]
As special cases of Theorem \ref{thm:1.1} when $(X,d)=(\mathbb{R}^d,\Vert\cdot\Vert)$,  we only need to prove that the Besicovitch constant of $(\mathbb{R}^d,\Vert\cdot\Vert)$ equals its  strict Hadwiger number. But this fact has been shown in \cite[Theorem 3.2]{Al21}. 
\end{proof}

\section{Radon measure in Euclidean spaces}\label{S4}

Since every finite measure has no such lower bounds greater than 1, a natural attempt to generalize Corollary \ref{cor:1.2} is to consider measures only satisfying $\mu(X)=\infty$. The following example tells us that condition \eqref{con:1.1}  can not be omitted.

\begin{ex}\label{ex:4.1}
	Let $\Vert\cdot\Vert$ be any norm on $\mathbb{R}^d$.
	For any $p>1$ and $\varepsilon>0$,
	 there exists a discrete measure $\mu$ on $(\mathbb{R}^d,\Vert\cdot\Vert)$  which satisfies the following conditions{\rm :}
	\begin{enumerate}[\rm(i)]
		\item  $\mu(\mathbb{R}^d)=\infty;$
		\item  $	\mu(\{x\in \operatorname {supp} (\mu): r\mapsto \mu(B(x,r))\; \text{is discontinuous}\})=\infty;$
		\item  $\inf\limits_{\Vert g \Vert_p=1}\Vert Mg\Vert_p\leq 1+\varepsilon.$
	\end{enumerate}
\end{ex}

\begin{proof}
	Given $x\in \mathbb{R}$,  we use $e_x$ to denote the point $(x,0,\cdots,0)$ on $\mathbb{R}^d$. For $t>1$ and $i\in \mathbb{N}$, consider $\mu= \frac{1}{t-1}\delta_{e_0}+\sum\limits_{i=1}^\infty t^{i-1}\delta_{e_i},$ where $\delta_{y}$ is the Dirac measure concentrated at the point $y\in \mathbb{R}^d$. Since (i) and (ii) follow immediately, we only need to verify (iii). Assume  $i\in \mathbb{N}^+.$ Then,  using the convexity of the ball, we have \begin{equation*}
		M_\mu\mathbf{1}_{\{e_0\}}(e_i)=\frac{\Vert\mathbf{1}_{\{e_0\}}\Vert_{1,\mu}}{\inf_{B\ni e_i;B\ni e_0}\mu(B)}=\frac{\frac{1}{t-1}}{\frac{1}{t-1}+\sum_{j=1}^{i}t^{j-1}}=t^{-i}.
	\end{equation*}
Hence, a simple calculation gives 
\begin{align*}
	\Vert M_\mu\mathbf{1}_{\{e_0\}}\Vert_{p,\mu}^p&=\frac{1}{t-1}(M_\mu\mathbf{1}_{\{e_0\}}(e_0))^p+\sum_{i=1}^\infty t^{i-1}(M_\mu\mathbf{1}_{\{e_0\}}(e_i))^p\\
	&=\frac{1}{t-1}+\sum_{i=1}^\infty t^{(1-p)i-1} =\frac{1}{t-1}+\frac{1}{t^p-t}.\\
\end{align*}As $t\to \infty,$ we get $\frac{\Vert M_\mu\mathbf{1}_{\{e_0\}}\Vert_{p,\mu}^p}{\Vert\mathbf{1}_{\{e_0\}}\Vert_{p,\mu}^p}=1+\frac{t-1}{t^p-t}\to 1.$ Thus fixing $p>1$ and $\varepsilon>0$, $\inf\limits_{\Vert g \Vert_p=1}\Vert Mg\Vert_p\leq \frac{\Vert M_\mu\mathbf{1}_{\{e_0\}}\Vert_{p,\mu}}{\Vert\mathbf{1}_{\{e_0\}}\Vert_{p,\mu}}\leq 1+\varepsilon$, when $t$ is large enough.
\end{proof}

Set $c_p=\left(1+\frac{1}{(p-1)H^*(\mathbb{R}^d,\Vert\cdot\Vert)}\right)^{1\slash p},$ where $\Vert\cdot\Vert$ is a norm on $\mathbb{R}^d$. As we mentioned before, \eqref{con:1.1} is satisfied for every absolutely continuous measure with respect to the Lebesgue measure,
it is reasonable to further ask whether $M_\mu f$ for every continuous measure must have  bound $c_p$. Surprisingly, for lower-dimensional Euclidean space, the answer  is affirmative.

We need the following basic fact:

\begin{lem}\label{lem:4.1}
	For $d=1,2,$ let $\mu$ be a Radon measure on $(\mathbb{R}^d,\Vert\cdot\Vert_2)$ that has no atoms, then $\{x\in  \operatorname {supp} (\mu): r\mapsto \mu(B(x,r))\; \text{is discontinuous} \}$ is countable.  In particular, $\mu(\{x\in \operatorname {supp} (\mu): r\mapsto \mu(B(x,r))\; \text{is discontinuous}\})=0$.
\end{lem}

	\begin{proof}
		We begin by observing that $\lim\limits_{r\to r_0^+}\mu(B^{o}(x,r))=\mu(B^{cl}(x,r_0))$ and $\lim\limits_{r\to r_0^-}\mu(B^o(x,r))=\mu(B^o(x,r_0))$. Thus the continuity of this function at $r_0$ is equivalent to measure zero on the spherical boundary $\{y:d(y,x)=r_0\}$. If in the case of one dimension, the result follows immediately.
	 Next we turn to the case of  two dimension.  To prove it, we claim that 
		 \begin{equation*}
			E=\{S:\text{ where $S$ is an one-dimensional sphere } \text{s.t.} \;\mu(S)>0\} 
		\end{equation*} is countable.  
		We shall do this by contradiction. Suppose, if possible, that $E$ is uncountable. Letting $E= \left(S_{\gamma}\right)_{\gamma \in \Gamma}$ and $x_{\gamma}$ be a center of $S_\gamma$, we obtain
		$
			E=\bigcup_{i\in \mathbb{Z}}\bigcup_{j\in \mathbb{N}}\bigcup_{s\in \mathbb{N}}E_{i,j,s},
		$ where
	\begin{align*}
	E_{i,j,s}
		\triangleq \{S_{\gamma}:2^{-i-1}\leq\mu(S_\gamma)<2^{-i};\; j\leq\operatorname {radii} S_\gamma<j+1;\;s\leq d(x_{\gamma},0)<s+1 \}.
	\end{align*}
	 Then  there exists at least uncountable spheres meeting requirements listed above for some $i_0,j_0,s_0$, that is, $E_{i_0,j_0,s_0} $ is uncountable. 	The next  observation is that the intersection of  two different spheres is at most two points. By continuity of $\mu$, any two different spheres are  disjoint in the sense of measure. Hence we can choose a subfamily $\{S_{\gamma_k}\}_{k\in \mathbb{N}}\subset E_{i_0,j_0,s_0} $ so that $\mu(B^{cl}(0,j_0+s_0+2))\geq \mu(\bigcup_{k}S_{\gamma_k})=\sum_{k\in\mathbb{N}}\mu(S_{\gamma_k})\geq \infty,$  which contradicts the assumption that $\mu$ is a Radon measure. Thus $E$ is countable. Since every sphere has only one center,  
$
	 	\{x\in \mathbb{R}^2: \text{$\exists r$  s.t. $\mu(\{y:d(y,x)=r\})>0$}\}
$
	 is also countable as desired.
	\end{proof}

\begin{re}
	Notice that the lemma cannot be extended to the high dimension. An example of a singular continuous measure shows the value of the right-hand item of \eqref{con:1.1} can take infinite. Our methods at present are not able to tackle the high dimension problem.
\end{re}

By corresponding results of Sullivan \cite[Proposition 23]{Su94}, we know that strict Hadwiger numbers in one-dimensional and two-dimensional Euclidean space are 2 and 5. Thus the lower bounds  are greater than $\left(1+\frac{1}{2(p-1)}\right)^{1\slash p}$ and $\left(1+\frac{1}{5(p-1)}\right)^{1\slash p}$ respectively, and this verifies Theorem \ref{thm:1.3}.  We ignore the details of this proof since it is the direct inference of Sullivan's proposition, Corollary \ref{cor:1.2} and Lemma \ref{lem:4.1}.

 We would like to mention that in one dimension we can improve the lower bound estimation to $\left(\frac{p}{p-1}\right)^{1\slash p}$. 
The core of our proof is exactly consistent with Lerner's approach under one-dimensional Lebesgue measures.
To this end, we introduce the following theorem proved by  Ephremidze et al.\cite[Theorem 1]{EFT07}.

\begin{thm}\label{thm:4.2}
 Let $\mu$ be a Borel semi-regular measure on $\mathbb{R}$ that contains no atoms. For $f\geq 0$, we consider the one sided maximal operator \begin{equation*}
 	M_+f(x)= \sup_{b>x}\frac{1}{\mu([x,b))}\int_{[x,b)}f(s)d\mu.
 \end{equation*} If 
$
 t > \liminf\limits_{x \to -\infty}M_+f(x), 
$ then \begin{equation*}
	t\mu(\{x\in\mathbb{R}:M_+f(x)>t\}) = \int_{\{x\in\mathbb{R}:M_+f(x)>t\}}f(s)d\mu.
\end{equation*}
\end{thm}
\begin{proof}[Proof of Theorem \ref{thm:1.4}.]
  If $\mu(\mathbb{R})=\infty$, we have either $\mu((0,\infty))=\infty$ or $\mu((-\infty,0])=\infty$. Without loss of generality, we assume $\mu((-\infty,0])=\infty$. Otherwise we consider 
  $
  M_-f(x)= \sup\limits_{a<x}\frac{1}{\mu((a,x])}\int_{(a,x]}f(s)d\mu
$  instead of $M_{+}f(x)$. As $\mu$ is Radon, $\mu((-\infty,-R))=\mu((-\infty,0])-\mu([-R,0])=\infty$ for any $R>0$. Assume that the function $f\in L^1(\mu)$ with $\operatorname {supp} (f)\subseteq[-R,R]$. If $x<-R$, we have 
  \begin{equation*}
  	M_+f(x)= \sup_{b>-R}\frac{1}{\mu([x,b))}\int_{[x,b)}f(s)d\mu\leq\sup\limits_{b>-R}\frac{\Vert f\Vert_1}{\mu([x,b))}\leq\frac{\Vert f\Vert_1}{\mu([x,-R))}.
  \end{equation*}From $\mu((-\infty, -R))=\infty$, we obtain 
\begin{equation*}
		\liminf\limits_{x \to -\infty}M_+f(x) =	\liminf\limits_{x \to -\infty}\frac{\Vert f\Vert_1}{\mu([x,-R))}= 0. 
\end{equation*} 
 Applying Theorem \ref{thm:4.2}, we conclude
\begin{equation}\label{eq:4.1}
 	t\mu(\{x\in\mathbb{R}:M_+f(x)>t\}) = \int_{\{x\in\mathbb{R}:M_+f(x)>t\}}f(s)d\mu \;\;\text{for all $t>0$}.
 \end{equation}
Now multiplying both sides of \eqref{eq:4.1} by $t^{p-2}$ and integrating with respect to $t$ over $(0,\infty)$, we obtain $\frac{\Vert M_+f\Vert_p^p}{p}=\frac{1}{p-1}\int_{\mathbb{R}}f(M_+f)^{p-1}d\mu$. Therefore, $\frac{\Vert Mf\Vert_p^p}{p}\geq\frac{1}{p-1}\int_{\mathbb{R}}f^pd\mu$ and the desired inequality follows. 
\end{proof}

For one-dimensional Banach space, there are more powerful covering theorems to ensure the weak type inequality with respect to arbitrary measures, as was done by Peter Sjogren in \cite{Sj83}. See also \cite{Be89} and \cite{GK98}. Hence the maximal operator $M$ satisfies
the strong type $L^p(\mu)$ estimates for $1<p<\infty$. The function family of the following criterion is easier to process in comparison with Lemma \ref{lem:3.4}.

\begin{lem}\label{lem:4.3}
	Let $\mu$ be a Borel semi-regular measure on $\mathbb{R}$ and let $C$ be a constant.  The following are equivalent: \begin{enumerate}[\rm(i)]
		\item $\Vert Mf\Vert_p\geq C \Vert f\Vert_p$ for all $f\in L^p(\mu)$.
		\item $\Vert Mg\Vert_p\geq C\Vert g\Vert_p$ for all  $g$ of form $g=\sum_{i=1}^N\beta_i\mathbf{1}_{I_i}$, where the bounded open intervals $I_i$ are disjoint and $\beta_i>0$.
	\end{enumerate}
\end{lem}

\begin{proof}
	Suppose that (ii) is invalid. 
	Let $f\in
	L^p(\mu)$. Since the sub-linear operator $M$ is bounded on $L^p(\mu)$, we have  $Mf_n \to  Mf$ as $f_n \to f$. Thus (i) follows from the fact that the set  of all positive coefficients linear combinations of  characteristic functions of bounded open intervals is dense. 
\end{proof}

	Let $A_\mu$ be the set of all the real numbers $x$ with
	$\mu(\{x\})>0$. If we denote by $x_n$ the points
	of $A_\mu$ and by $w_n$ the weight of $x_n$, the decomposition follows: $\mu =\mu_c+\sum_{n}w_n\delta_{x_n}$, where the continuous part of $\mu$: $\mu_c$ is defined by
	$\mu_c(B)=\mu(B\cap A^c)$. This brings us nicely to consider the case of that $\mu_c(\mathbb{R})=\infty.$

\begin{thm}\label{thm:4.4} Let $\mu$ be a Borel semi-regular measure on $\mathbb{R}$ and suppose that the axis of negative and positive reals have measures infinite.
	If one of the following holds:
	\begin{enumerate}[\rm(i)]
	\item 	the set $A_\mu$ is one point;
		\item the set $A_\mu$ contains only two points and there is no mass between these two points; 
	\end{enumerate} then 
	\[
	\Vert M_\mu f \Vert_p \geq \left(\frac{p}{p-1}\right)^{\frac{1}{p}}\Vert f\Vert_p \quad\text{for all $f\in L^p(\mu)$}.
	\]
\end{thm}
\begin{proof}
	We claim that the ball-coverings family assumed in Theorem \ref{thm:3.6} exists for each step function, then we justify the theorem.
		Let $f=\sum_{i=1}^N\beta_i\mathbf{1}_{I_i}$, where the bounded open intervals $I_i$ are disjoint and $\beta_i>0$. If necessary, we rearrange the intervals $I_i$ to ensure that it is mutually disjoint successively. Meanwhile, we also relabel the series $\beta_i$, that is $\beta_{i_1}\leq \beta_{i_2}\leq \dots \leq \beta_{i_N}$, where $(i_1,\dots i_N)$ is a permutation of $(1,\dots, N)$. Suppose first (i) holds and write $A_\mu=\{y\}$.
	
We first consider the case when $y\notin \cup_{1\leq i\leq N}I_i$. Now suppose $0<t<\beta_{i_N}$,  otherwise the corresponding superlevel set is nonempty. We define $j_t$ to be the first number in $\{\beta_{i_j}\}_{1\leq j\leq N}$ skipping over $t$, then $\{x\in X: f(x)>t\}= \cup_{j\geq j_t}I_{i_j}.$ Now we apply a standard selection
	procedure: since $y\notin \cup_{1\leq i\leq N}I_i$, $(-\infty,y]$ and $(y,+\infty)$ cut apart the intervals family $\cup_{j\geq j_t}I_{i_j}$ into the two parts, then we write $b_1$ be the right endpoint of the last interval in the list which lies on the left side of $y$ (the exception is considered later); thus there exists a ball $B_1=(s_1,b_1)$ such that $\langle f\rangle_{B_1}=t$, since $\langle f\rangle_{(s,b_1)}$ is continuous, $\lim_{s\to b_1}\langle f\rangle_{(s,b_1)}>t$ and $\lim_{s\to -\infty}\langle f\rangle_{(s,b_1)}=0$ which all stems from the assumption that $\mu|_{[-\infty,y)}$ is an infinite continuous measure; if $s_1\notin \cup_{j\geq j_t}I_{i_j}$, we  choose it; otherwise, using a while loop,  we will find a new ball $B_s=(s_{i},s_{i-1})$ to meet $\langle f\rangle_{B_s}=t$ until $s_i\notin \cup_{j\geq j_t}I_{i_j}$, then we select $(s_i,b_1)$ be the ball $B_1$ as desired; having disposed of this step, we use induction to find the family of balls which covers  the part of $\cup_{j\geq j_t}I_{i_j}$ belonging to $(-\infty,y]$. For the intervals family on the right side of $y$, we do the same by consider the ball in $(y,+\infty)$. It was clear that the balls chosen was what we needed.
	
	We now turn to the case when $y\in \cup_{1\leq i\leq N}I_i$. Then we suppose $I_{i_y}\ni y$ for some $i_y\in (1,\dots, N)$. The proof of this case is quite similar to the former  since we can choose the right endpoint of $I_{i_y}$ as the starting point of the selection
	procedure.
	
	The same reasoning applies to the case (ii).
\end{proof}

The following example indicates that a one-sided infinite measure containing only one atom can also lead to the phenomenon of Example \ref{ex:4.1}.

\begin{ex}
	For $t>1$, consider $\mu= t\delta_{1}+m|_{(0,\infty)},$ where $m$ is the Lebesgue measure. Obviously, the examples satisfy $\mu(\mathbb{R})=\infty$. Letting $f=\mathbf{1}_{(0,1)}$, we will show that we can take large $t $ such that $\Vert M_{\mu}f\Vert_{p,\mu}$ gets close enough to $1$. If $x\in (0,1)$, $M_\mu f(x)=1$; if  $x\in [1,\infty)$, then
\begin{equation*}
	M_\mu 
	f(x)=\sup\limits_{a<x}\frac{1}{\mu((a,x])}\int_{(a,x]}f(s)d\mu=\sup\limits_{0<a\leq1}\frac{1-a}{t+x-a}=\frac{1}{t+x}.
\end{equation*}
Thus, we have
\begin{align*}
	\Vert M_{\mu}f \Vert^p_{p,\mu}=&\Vert M_{\mu}f \Vert^p_{p,m|_{(0,\infty)}}+\Vert M_{\mu}f \Vert^p_{p,t\delta_1}\\
	=&1+\int_{1}^{\infty}\frac{1}{(t+x)^p}dx+\frac{t}{(t+1)^p} \\
	\leq& 1+\frac{p}{p-1} (t+1)^{1-p}.
\end{align*}
Since $\lim\limits_{t\to \infty}\Vert M_{\mu}f\Vert_{p,\mu} =1$, the result follows.
\end{ex}

\begin{re}
	Another similar example shows that in general Theorem \ref{thm:4.4} cannot be popularizing to the case where $A_\mu$ contains more than three atoms.
\end{re}

Finally, we comment on a result that the lower $L^p$-bound of $M$ becomes significantly increased with the decreases of $p$.

% The result shows that if the lower $L^r$-bound of the uncentered Hardy-Littlewood maximal operator is strictly greater than $1$, then so is $L^p$-bound  for all $1<p<r$. 
Since it is immediate consequence of H{\"o}lder's inequality, the proof is not shown here.
%Let $(X,d,\mu)$ be a metric measure
%space and $p>1$. If $f\in L^p_{loc}(\mu)$, then $ \left(Mf(x)\right)^p \leq M(|f|^p)(x)$.
%\end{prop}
%\begin{proof}
	%Let $q=\frac{p}{p-1}$. Then applying H{\"o}lder's inequality, we have \begin{equation*}
	%	\langle f\rangle_{B} \leq \frac{\left(\int_{B}|f|^pd\mu\right)^{1/p}\left(\int_{B}1^qd\mu\right)^{1/q}}{\mu(B)}\leq\left(\frac{\int_{B}|f|^pd\mu}{\mu(B)}\right)^{\frac{1}{p}}\leq \left(\langle |f|^p\rangle_{B}\right)^{\frac{1}{p}}.
%	\end{equation*}
%Therefore 
%	\begin{equation*}
%	Mf(x)=\sup_{x\in B}\langle f\rangle_{B}\leq \sup_{x\in B}\left(\langle |f|^p\rangle_{B}\right)^{\frac{1}{p}}\leq \left(\sup_{x\in B}\langle |f|^p\rangle_{B}\right)^{\frac{1}{p}}\leq \left(M(|f|^p)(x)\right)^{\frac{1}{p}}.
%	\end{equation*}
% Exponentiating both sides by $p$, the proposition follows.
%\end{proof}
\begin{prop}
Let $(X,d,\mu)$ be a metric measure
space and $M_r = \inf\limits_{\Vert g \Vert_r=1}\Vert Mg\Vert_r$, then $M_r \leq M_p^{\frac{p}{r}}$ for $1<p<r$.
\end{prop}
%\begin{proof}
%	Suppose first $r<\infty$, otherwise the proposition  follows by $\Vert M(f)\Vert_{\infty}\leq\Vert f \Vert_{\infty}.$ If $\Vert f\Vert_{p}=1$, we now claim that there exists a function $g$ with $\Vert g\Vert_{r}=1$ and $\Vert M g\Vert_r\leq \Vert Mf\Vert^{\frac{p}{r}}_p$. In fact, let $g=|f|^{\frac{p}{r}}$, then
%	$\Vert g\Vert^r_r = \int g^rd\mu =\Vert f \Vert^p_p$, and hence  $\Vert g\Vert_{r}=1$. Then for $\frac{r}{p}> 1$, using obtained proposition \ref{pr:1}, we get
%	\begin{equation*}
%		Mf(x)=M(g^\frac{r}{p})(x)\geq \left(Mg(x) \right)^{\frac{r}{p}}.
%	\end{equation*}
%Therefore 
%	\begin{equation*}
%	\Vert Mf\Vert_p=\left(\int(Mf)^p\right)^{\frac{1}{p}}\geq \left(\int(Mg)^r\right)^{\frac{1}{p}}\geq\Vert Mg\Vert^{\frac{r}{p}}_r.
%	\end{equation*}
%The proposition follows by combing these and the definition of $M_r$.
%\end{proof} 

\end{document}